\documentclass[11pt]{amsart}%
\usepackage{amsfonts}
\usepackage{amsmath}
\usepackage{amssymb}
\usepackage{graphicx}%
\providecommand{\U}[1]{\protect\rule{.1in}{.1in}}
\newtheorem{theorem}{Theorem}
\theoremstyle{plain}

\newtheorem{corollary}{Corollary}

\newtheorem{definition}{Definition}

\newtheorem{lemma}{Lemma}

\newtheorem{proposition}{Proposition}

\numberwithin{equation}{section}

\newcommand{\beqa}{\begin{eqnarray*}}
\newcommand{\eeqa}{\end{eqnarray*}}

\def\<{\left<}
\def\>{\right>}

\def\mv1{M_v^1}

\def\Ren{\mathbb{R}^d}

\def\Sn2{S_{2}(L^{2}(\Ren))}
\def\S1{S_{1}(L^{2}(\Ren))}
\def\sig00{\sigma_{0,0}}

\begin{document}

\title[Born-Jordan Pseudo-Differential Operators]{Born-Jordan Pseudo-Differential Operators with Symbols in the Shubin Classes}
\author{Elena Cordero}
\address{Dipartimento di Matematica, Universit\`a di Torino, Dipartimento di
Matematica, via Carlo Alberto 10, 10123 Torino, Italy}
\email{elena.cordero@unito.it}
\author{Maurice de Gosson}
\address{University of Vienna, Faculty of Mathematics (NuHAG), Oskar-Morgenstern-Platz
1, 1090 Vienna}
\email{maurice.de.gosson@univie.ac.at}
\author{Fabio Nicola}
\address{Dipartimento di Scienze Matematiche, Politecnico di Torino, corso Duca degli
Abruzzi 24, 10129 Torino, Italy}
\email{fabio.nicola@polito.it}
\date{}

\begin{abstract}
We apply Shubin's theory of global symbol classes $\Gamma_{\rho}^{m}$ to the
Born-Jordan pseudodifferential calculus we have previously developed. This
approach has many conceptual advantages, and makes the relationship between
the conflicting Born-Jordan and Weyl quantization methods much more limpid.
We give, in particular, precise asymptotic expansions of symbols allowing to
pass from Born-Jordan quantization to Weyl quantization, and vice-versa. In
addition we state and prove some regularity and global hypoellipticity results.

\end{abstract}
\maketitle


\section{Introduction}

The Born-Jordan quantization rules \cite{bj,bjh,Cohen1} have recently been
rediscovered in mathematics and have quickly become a very active area of
research under the impetus of scientists working in signal theory and
time-frequency analysis \cite{bogetal,bodeol2013,Cohenbook,cogoni16b}. It has been
realized, not only that the associated phase space picture has many advantages
compared with the usual Weyl--Wigner picture (it allows a strong damping of
unwanted interference patterns \cite{bogetal,cogoni16a}), but also, as one of
us has shown \cite{Springer,FPBJ,PRBJ}, that there is strong evidence that
Born-Jordan quantization might very well be the correct quantization method
in quantum physics. Independently of these potential applications, the
Born-Jordan pseudo-differential calculus has many interesting and difficult
features (some of them, as non-injectivity \cite{cogoni15}, being even quite
surprising) and deserve close attention. The involved mathematics is less
straightforward than that of the usual Weyl formalism; for instance
Born-Jordan pseudo-differential calculus is not fully covariant under linear
symplectic transformations \cite{TRANSAM}, which makes the study of the
symmetries of the operators much less straightforward than in the Weyl case.

In the present paper we set out to study the
pseudo-differential calculus associated with Born-Jordan quantization in the framework of
Shubin's \cite{Shubin} global symbol classes. These results complement and
extend those obtained by us in \cite{cogoni15}. \par
To be precise, in the Weyl quantization scheme to any observable (symbol) $a(z)$, $z\in \mathbb{R}^{2n}$, defined as a function or (temperate) distribution in phase space, it is associated the Weyl operator
\begin{equation*}
\widehat{A}_{\mathrm{W}}=\left(  \tfrac{1}{2\pi\hbar}\right)  ^{n}\int
a_{\sigma}(z)\widehat{T}(z)d^{2n}z 
\end{equation*}
where $a_{\sigma}=F_{\sigma}a$ is the symplectic Fourier transform of $a$
and $\widehat{T}(z_{0})$ is the
Heisenberg operator given by
\begin{equation*}
\widehat{T}(z_{0})\psi(x)=e^{\tfrac{i}{\hslash}(p_{0}x-\tfrac{1}{2}p_{0}%
x_{0})}\psi(x-x_{0}).
\end{equation*}
This is simply a phase space shift and, as a consequence of the Schwartz kernel theorem, every continuous linear operator $\mathcal{S}(\mathbb{R}^{n})\to \mathcal{S}'(\mathbb{R}^{n})$ can be written in a unique way as a Weyl operator for a suitable symbol $a\in\mathcal{S}'(\mathbb{R}^{2n})$; namely, it is a superposition of phase space shifts. In this functional framework the Weyl correspondence between observables and operators is therefore one to one. \par
The Born-Jordan quantization of a symbol $a(z)$ is instead defined  as 
\[
\widehat{A}_{\mathrm{BJ}}=\left(  \tfrac{1}{2\pi\hbar}\right)  ^{n}\int
a_{\sigma}(z)\operatorname{sinc}\left(  \tfrac{px}{2\hbar}\right)  \widehat
{T}(z)d^{2n}z
\]
with $z=(x,p)$ and $px=p\cdot x$. The presence of the function $\operatorname{sinc}\left(  \tfrac{px}{2\hbar}\right)$ and in particular its zeros make the corresponding quantization problem much more subtle. It was proved in \cite{cogoni15} that every linear continuous operator $\mathcal{S}(\mathbb{R}^{n})\to\mathcal{S}'(\mathbb{R}^{n})$ can still be written in Born-Jordan form, but the representation is no longer unique. The Born-Jordan correspondence is anyway still surjective.\par
 In this paper we continue this investigation by focusing on a particulary relevant subclass of  smooth symbols satisfying good growth conditions at infinity, namely Shubin's classes \cite{Shubin}. Roughly speaking the main result reads as follows. Within such symbol classes the Weyl symbol $a_{\rm W}$ and the corresponding Born-Jordan symbol $a_{\rm BJ}$ are related by the following explicit asymptotic expansions: 
 \[
a_{\rm W}(x,p)\sim\sum_{\alpha\in\mathbb{N}^{n}\atop |\alpha|\ {\rm even}}\frac{1}{\alpha!(|\alpha|+1)}\left(
\frac{i\hbar}{2}\right)  ^{|\alpha|}\partial_{x}^{\alpha}\partial_{p}^{\alpha
}a_{\rm BJ}(x,p)%
\]
and
\[
a_{\rm BJ}(x,p)\sim\sum_{\alpha\in\mathbb{N}^{n}\atop |\alpha|\ {\rm even}}\frac{c_{\alpha}}{\alpha!}\left(
\frac{i\hbar}{2}\right)  ^{|\alpha|}\partial_{x}^{\alpha}\partial_{p}^{\alpha
}a_{\rm W}(x,p)%
\]
for suitable coefficients $c_\alpha$ (see \eqref{calpha} below).\par
These expansions seem remarkable, because at present there is no an exact and explicit formula for the Born-Jordan symbol corresponding to a given Weyl operator, although the {\it existence} of such a symbol was proved in \cite{cogoni15}. Indeed, the situation seems definitely similar to what happens in the division problem of temperate distributions by a (not identically zero) polynomial $P$: the map $f\mapsto Pf$ from $\mathcal{S}'(\mathbb{R}^{n})$ into itself is onto but in general does not exist a linear continuous right inverse \cite{bonet,lang}.\par
We refer to \cite{Treschev} for an alternative formulation of quantum observables by formal series of noncommutative generators $\widehat{x}, \widehat{p}$.
\medskip

We will systematically use properties of the global pseudo-differential
calculus whose study was initiated by Shubin, after related work by Beals,
Berezin, Kumano-go, Rabinovi\v{c}, and others (see the Bibliography in
\cite{Shubin}). This calculus plays an important role in quantum mechanics
since the position and momentum variables are placed on an equal footing in
the estimates defining the symbol classes. 
We have found this approach particularly well adapted to investigate
asymptotic expansions such as those for $a_{W}$ and $a_{BJ}$. 

Natural related topics that we have not included in this work are the spectral theory of Born-Jordan operators, in which the notion of global hypoellipticity plays a crucial role and the anti-Wick version of these operators (the latter might lead to some new insights). Finally, we have not discussed at all the
Wigner--Moyal formalism associated with Born-Jordan question; for the latter
we refer to \cite{bogetal,cogoni16a,Springer}.

 \par\medskip
In short, the paper is organized as follows. In Section 2 we review the definition of the Born-Jordan pseudodifferential operators. Section 3 is devoted to Shubin's symbol classes. In Section 4 we prove the above relationships between Weyl and Born-Jordan symbol. Finally Section 5 is devoted to applications to the global regularity problem.\par\medskip

{\bf Notation.}
We denote by $\sigma$ the standard symplectic form $\sum_{j=1}^{n}dp_{j}\wedge
dx_{j}$ on the phase space $\mathbb{R}^{2n}\equiv\mathbb{R}^{n}\times
\mathbb{R}^{n}$; the phase space variable is written $z=(x,p)$. Equivalently,
$\sigma(z,z^{\prime})=Jz\cdot z^{\prime}$ where $J=%
\begin{pmatrix}
0 & I\\
-I & 0
\end{pmatrix}
$. We will denote by $\widehat{x}_{j}$ the operator of multiplication by
$x_{j}$ and set $\widehat{p}_{j}=-i\hbar\partial/\partial x_{j}$. These
operators satisfy Born's canonical commutation relations $[\widehat{x}%
_{j},\widehat{p}_{j}]=i\hbar$ where $\hbar$ is a positive parameter such that
$0<\hbar$ $\leq1$.
\par\medskip

\section{Born-Jordan pseudo-differential Operators\label{sec1}}

In this section we review the recent advances in the theory of Born-Jordan
quantization; for proofs and details we refer to Cordero et al.
\cite{cogoni15}, de Gosson \cite{TRANSAM,Springer,FPBJ}.

\subsection{The Born-Jordan quantization rules}

Following Heisenberg's insightful work on \textquotedblleft matrix
mechanics\textquotedblright\ Born and Jordan \cite{bj,bjh} proposed the
quantization rule%
\begin{equation}
p^{s}x^{r}\overset{\mathrm{BJ}}{\longrightarrow}\frac{1}{s+1}\sum_{\ell=0}%
^{s}\widehat{p}^{s-\ell}\widehat{x}^{r}\widehat{p}^{\ell} \label{bj1}%
\end{equation}
for monomials. Their rule conflicts with Weyl's \cite{Weyl} quantization rule,
leading to
\begin{equation}
p^{s}x^{r}\overset{\mathrm{W}}{\longrightarrow}\frac{1}{2^{s}}\sum_{\ell
=0}^{s}\binom{s}{\ell}\widehat{p}^{s-\ell}\widehat{x}^{r}\widehat{p}^{\ell},
\label{w2}%
\end{equation}
(McCoy rule \cite{mccoy}) as soon as $r\geq2$ and $s\geq2$. The following
observation is crucial: both quantizations are obtained from Shubin's $\tau
$-rule
\begin{equation}
p^{s}x^{r}\overset{\tau}{\longrightarrow}\sum_{\ell=0}^{s}\binom{s}{\ell
}(1-\tau)^{\ell}\tau^{s-\ell}\widehat{p}^{s-\ell}\widehat{x}^{r}\widehat
{p}^{\ell} \label{tau1}%
\end{equation}
but by very different means. In fact, the Weyl rule (\ref{w2}) is directly
obtained by choosing $\tau=\frac{1}{2}$ while Born and Jordan's rule (\ref{bj1}) is
obtained  by averaging the right-hand side of (\ref{tau1}) with
respect to $\tau$ over the interval $[0,1]$ (de Gosson and Luef \cite{golu1},
de Gosson \cite{TRANSAM,Springer}).

On the operator level, the Weyl operator $\widehat{A}_{\mathrm{W}%
}=\operatorname*{Op}_{\mathrm{W}}(a)$ is given by the familiar formula due to
Weyl himself \cite{Weyl}
\begin{equation}
\widehat{A}_{\mathrm{W}}=\left(  \tfrac{1}{2\pi\hbar}\right)  ^{n}\int
a_{\sigma}(z)\widehat{T}(z)d^{2n}z \label{AW1}%
\end{equation}
where $a_{\sigma}=F_{\sigma}a$ is the symplectic Fourier transform
\begin{equation}
a_{\sigma}(z)=\left(  \tfrac{1}{2\pi\hbar}\right)  ^{n}\int e^{-\frac{i}%
{\hbar}\sigma(z,z^{\prime})}a(z^{\prime})d^{2n}z^{\prime} \label{sympft}%
\end{equation}
and $\widehat{T}(z_{0})=e^{-\frac{i}{\hslash}\sigma(\hat{z},z_{0})}$ is the
Heisenberg operator; recall \cite{Birk,75} that the action of $\widehat
{T}(z_{0})$ on a function or distribution $\psi$ is explicitly given by%
\begin{equation}
\widehat{T}(z_{0})\psi(x)=e^{\tfrac{i}{\hslash}(p_{0}x-\tfrac{1}{2}p_{0}%
x_{0})}\psi(x-x_{0}). \label{heiwe}%
\end{equation}
Let us underline that the parameter $\hslash\in(0,1]$ is fixed in our context.
Here we are not interested in the semi-classical analysis, i.e.\ the asymptotic
as $\hslash\rightarrow0$.

Using Plancherel's identity formula (\ref{AW1}) can be rewritten%
\begin{equation}
\widehat{A}_{\mathrm{W}}=\left(  \tfrac{1}{\pi\hbar}\right)  ^{n}\int
a(z)\widehat{\Pi}(z)d^{2n}z \label{awgr1}%
\end{equation}
where
\begin{equation}
\widehat{\Pi}(z)=\widehat{T}(z)\widehat{\Pi}\widehat{T}(z)^{-1}
\label{groroydef}%
\end{equation}
is the Grossmann--Royer reflection operator (where $\widehat{\Pi}\psi
(x)=\psi(-x)$). One verifies that under suitable convergence conditions (for
instance $a\in\mathcal{S}(\mathbb{R}^{2n})$ and $\psi\in L^{1}(\mathbb{R}%
^{n})$) one recovers the more familiar \textquotedblleft mid-point
formula\textquotedblright%
\begin{equation}
\widehat{A}_{\mathrm{W}}\psi(x)=\left(  \tfrac{1}{2\pi\hbar}\right)  ^{n}\iint
e^{\frac{i}{\hbar}p(x-y)}a(\tfrac{1}{2}(x+y),p)\psi(y)d^{n}yd^{n}p
\label{weylop}%
\end{equation}
common in the theory of pseudo-differential operators; we will use this
notation as a formal tool for the sake of clarity (keeping in mind that it can
be given a rigorous meaning by (\ref{awgr1})). The easiest way to define
Shubin's $\tau$-operator $\widehat{A}_{\tau}=\operatorname*{Op}_{\tau}(a)$ is
to use the formula above as a starting point, and to replace the midpoint
$\tfrac{1}{2}(x+y)$ with $(1-\tau)x+\tau y$ which leads to
\begin{equation}
\widehat{A}_{\tau}\psi(x)=\left(  \tfrac{1}{2\pi\hbar}\right)  ^{n}\iint
e^{\frac{i}{\hbar}p(x-y)}a((1-\tau)x+\tau y,p)\psi(y)d^{n}yd^{n}p.
\label{shubinop}%
\end{equation}
As in the monomial case, the Born-Jordan operator $\widehat{A}_{\mathrm{BJ}%
}=\operatorname*{Op}_{\mathrm{BJ}}(a)$ is obtained by averaging
(\ref{shubinop}) over $[0,1]$:%
\begin{equation}
\widehat{A}_{\mathrm{BJ}}=\int_{0}^{1}\widehat{A}_{\tau}d\tau. \label{2.10bis}%
\end{equation}

\subsection{Harmonic representation of Born-Jordan operators}

The following result gives an explicit expression of the Weyl symbol of a
Born-Jordan operator with arbitrary symbol (see \cite{bogetal,cogoni15}):

\begin{proposition}
Let $a\in\mathcal{S}^{\prime}(\mathbb{R}^{2n})$. (i) The operator $\widehat
{A}_{\mathrm{BJ}}=\operatorname*{Op}_{\mathrm{BJ}}(a)$ is the Weyl operator
$\operatorname*{Op}_{\mathrm{W}}(b)$ where%
\begin{equation}
b(x,p)=\left(  \tfrac{1}{2\pi\hbar}\right)  ^{n}(a\ast\theta)(x,p)
\label{baxp}%
\end{equation}
where $\theta\in S^{\prime}(\mathbb{R}^{2n})$ is the distribution whose
(symplectic) Fourier transform is
\begin{equation}
\theta_{\sigma}(x,p)=\operatorname{sinc}\left(  \frac{px}{2\hbar}\right)  .
\label{thetasigma}%
\end{equation}
(ii) The restriction of $\widehat{A}_{\mathrm{BJ}}$ to monomials $p_{j}%
^{s}x_{j}^{r}$ is given by the Born-Jordan rule (\ref{bj1}).
\end{proposition}

Recall that the function $\operatorname{sinc}$ is defined by
$\operatorname{sinc}u=\sin u/u$ for $u\neq0$ and $\operatorname{sinc}0=1$.

It follows from (\ref{AW1}) and the convolution formula $F_{\sigma}%
(a\ast\theta)=(2\pi\hbar)^{n}a_{\sigma}\theta_{\sigma}$that $\widehat
{A}_{\mathrm{BJ}}$ is alternatively given by
\begin{equation}
\widehat{A}_{\mathrm{BJ}}=\left(  \tfrac{1}{2\pi\hbar}\right)  ^{n}\int
a_{\sigma}(z)\operatorname{sinc}\left(  \tfrac{px}{2\hbar}\right)  \widehat
{T}(z)d^{2n}z\label{abjharmonic}%
\end{equation}
(cf. formula (\ref{AW1}) for Weyl operators).

\section{Symbol Classes}

In what follows we use the notation $\left\langle u\right\rangle
=\sqrt{1+|u|^{2}}$ for\ $u\in\mathbb{R}^{m}$. For instance, if $z=(x,p)\in
\mathbb{R}^{2n}$ then%
\[
\left\langle z\right\rangle =\sqrt{1+|z|^{2}}=\sqrt{1+|x|^{2}+|p|^{2}}.
\]
We assume that the reader is familiar with multi-index notation: if
$u=(u_{1},...,u_{m})\in\mathbb{R}^{m}$ and $\alpha=(\alpha_{1},...,\alpha
_{m})\in\mathbb{N}^{m}$ we write $u^{\alpha}=u_{1}^{\alpha_{1}}\cdot\cdot\cdot
u_{m}^{\alpha_{m}}$; similarly $\partial_{u}^{\alpha}=\partial_{u_{1}}%
^{\alpha_{1}}\cdot\cdot\cdot\partial_{u_{m}}^{\alpha_{m}}$. By definition
$|\alpha|=\alpha_{1}+\cdot\cdot\cdot+\alpha_{m}$ and $\alpha!=\alpha_{1}%
!\cdot\cdot\cdot\alpha_{m}!$. \ 

\subsection{The Shubin symbol class $\Gamma_{\rho}^{m}$}

We begin by giving the following definition (Shubin \cite{Shubin}, Definition 23.1):

\begin{definition}
Let $m\in\mathbb{R}$ and $0<\rho\leq1$. The symbol class $\Gamma_{\rho}%
^{m}(\mathbb{R}^{2n})$ consists of all complex functions $a\in C^{\infty
}(\mathbb{R}^{2n})$ such that for every $\alpha\in\mathbb{N}^{2n}$ there
exists a constant $C_{\alpha}\geq0$ with
\begin{equation}
|\partial_{z}^{\alpha}a(z)|\leq C_{\alpha}\left\langle z\right\rangle
^{m-\rho|\alpha|}\text{ \ for }z\in\mathbb{R}^{2n}. \label{est1}%
\end{equation}

\end{definition}

It immediately follows from this definition that if $a\in\Gamma_{\rho}%
^{m}(\mathbb{R}^{2n})$ and $\alpha\in\mathbb{N}^{2n}$ then $\partial
_{z}^{\alpha}a\in\Gamma_{\rho}^{m-\rho|\alpha|}(\mathbb{R}^{2n})$; using Leibniz's
rule for the derivative of products of functions one easily checks that
\begin{equation}
a\in\Gamma_{\rho}^{m}(\mathbb{R}^{2n})\text{ \textit{and} }b\in\Gamma_{\rho
}^{m^{\prime}}(\mathbb{R}^{2n})\Longrightarrow ab\in\Gamma_{\rho}%
^{m+m^{\prime}}(\mathbb{R}^{2n}). \label{gammatau2}%
\end{equation}

The class $\Gamma_{\rho}^{m}(\mathbb{R}^{2n})$ is a complex vector space for
the usual operations of addition and multiplication by complex numbers, and we
have
\begin{equation}
\Gamma_{\rho}^{-\infty}(\mathbb{R}^{2n})=\bigcap\nolimits_{m\in\mathbb{R}%
}\Gamma_{\rho}^{m}(\mathbb{R}^{2n})=\mathcal{S}(\mathbb{R}^{2n}).
\label{gammaminf}%
\end{equation}

The reduced harmonic oscillator Hamiltonian $H(z)=\frac{1}{2}(|x|^{2}%
+|p|^{2})$ obviously belongs to $\Gamma_{1}^{2}(\mathbb{R}^{2n})$, and so does%
\[
H(z)=\sum_{j=1}^{n}\frac{1}{2m_{j}}(p_{j}^{2}+m_{j}^{2}\omega_{j}^{2}x_{j}%
^{2});
\]
in fact, any polynomial function in $z$ of degree $m$ is in $\Gamma_{1}%
^{m}(\mathbb{R}^{2n})$. In particular every Hamiltonian function of the type
\[
H(z)=\sum_{j=1}^{n}\frac{1}{2m_{j}}p_{j}^{2}+V(x)
\]
belongs to some class $\Gamma_{1}^{m}(\mathbb{R}^{2n})$ if the potential
function $V(x)$ is a polynomial of degree $m\geq 2$.

The following lemma shows that the symbol classes $\Gamma_{\rho}%
^{m}(\mathbb{R}^{2n})$ are invariant under linear automorphisms of phase space
(this property does not hold for the usual H\"{o}rmander classes
$S_{\rho,\delta}^{m}(\mathbb{R}^{n})$ \cite{69}, whose elements are
characterized by growth properties in only the variable $p$). Let us denote by
$GL(2n,\mathbb{R})$ the space of $2n\times2n$ invertible real matrices. Then

\begin{lemma}
\label{Lemma1}Let $a\in\Gamma_{\rho}^{m}(\mathbb{R}^{2n})$ and $M\in
GL(2n,\mathbb{R})$. We have $a(M \cdot)\in\Gamma_{\rho}^{m}(\mathbb{R}^{2n})$.
\end{lemma}

\begin{proof}
The result is showed in greater generality in \cite[p. 177]{Shubin}. This
special case simply follows by the fact
\[
C^{-1}|z|\leq|Mz|\leq C|z|,
\]
for a suitable $C>0$.
\end{proof}

\subsection{Asymptotic expansions of symbols}

Let us recall the notion of asymptotic expansion of a symbol $a\in\Gamma
_{\rho}^{m}(\mathbb{R}^{2n})$ (cf. \cite{Shubin}, Definition 23.2):

\begin{definition}
\label{23.2}Let $(a_{j})_{j}$ be a sequence of symbols $a_{j}\in\Gamma_{\rho
}^{m_{j}}(\mathbb{R}^{2n})$ such that $\lim_{j\rightarrow+\infty}%
m_{j}\rightarrow-\infty$. Let $a$ $\in C^{\infty}(\mathbb{R}^{2n})$. If for
every integer $r\geq2$ we have%
\begin{equation}
a-\sum_{j=1}^{r-1}a_{j}\in\Gamma_{\rho}^{\overline{m}_{r}}(\mathbb{R}^{2n})
\label{23.4}%
\end{equation}
where $\overline{m}_{r}=\max_{j\geq r}m_{j}$ we will write $a\thicksim
\sum_{j=1}^{\infty}a_{j}$ and call this relation an asymptotic expansion of
the symbol $a$.
\end{definition}

The interest of the asymptotic expansion comes from the fact that every
sequence of symbols $(a_{j})_{j}$ with $a_{j}\in\Gamma_{\rho}^{m_{j}%
}(\mathbb{R}^{2n})$, the degrees $m_{j}$ being strictly decreasing and such
that $m_{j}\rightarrow-\infty$, determines a symbol in some $\Gamma_{\rho}%
^{m}(\mathbb{R}^{2n})$, that symbol being unique up to an element of
$\mathcal{S}(\mathbb{R}^{2n})$:

\begin{proposition}
\label{23.1.} Let $(a_{j})_{j}$ be a sequence of symbols $a_{j}\in\Gamma
_{\rho}^{m_{j}}(\mathbb{R}^{2n})$ such that $m_{j}>m_{j+1}$ and $\lim
_{j\rightarrow+\infty}m_{j}\rightarrow-\infty$. Then:

(i) \textit{There exists a function} $a$, \textit{such that} $a\thicksim
\sum\limits_{j=1}^{\infty}a_{j}$.

(ii) \textit{If another function} $a^{\prime}$ \textit{is such that
}$a^{\prime}\thicksim\sum\limits_{j=1}^{\infty}a_{j}$\textit{, then}
$a-a^{\prime}\in\mathcal{S}(\mathbb{R}^{2n})$.
\end{proposition}

(See Shubin \cite{Shubin}, Proposition 23.1). Note that property (ii)
immediately follows from (\ref{gammaminf}).

\subsection{The amplitude classes $\Pi_{\rho}^{m}$}

We will need for technical reasons an extension of the Shubin classes
$\Gamma_{\rho}^{m}(\mathbb{R}^{2n})$ defined above. Since Born-Jordan
operators are obtained by averaging Shubin's $\tau$-operators%
\[
\widehat{A}_{\tau}\psi(x)=\left(  \tfrac{1}{2\pi\hbar}\right)  ^{n}\iint
e^{\frac{i}{\hbar}p(x-y)}a((1-\tau)x+\tau y,p)\psi(y)d^{n}yd^{n}p
\]
over $\tau\in\lbrack0,1]$ we are led to consider pseudo-differential operators
of the type
\begin{equation}
\widehat{A}\psi(x)=\left(  \tfrac{1}{2\pi\hbar}\right)  ^{n}\iint e^{\frac
{i}{\hbar}p(x-y)}b(x,y,p)\psi(y)d^{n}yd^{n}p \label{aoscille}%
\end{equation}
where the function
\[
b(x,y,p)=\int_{0}^{1}a((1-\tau)x+\tau y,p)d\tau;
\]
is called \emph{amplitude} and is defined, not on $\mathbb{R}^{2n}%
\equiv\mathbb{R}_{x}^{n}\times\mathbb{R}_{p}^{n}$ but rather on $\mathbb{R}%
^{3n}\equiv\mathbb{R}_{x}^{n}\times\mathbb{R}_{y}^{n}\times\mathbb{R}_{p}^{n}%
$. It therefore makes sense to define an amplitude class generalizing
$\Gamma_{\rho}^{m}(\mathbb{R}^{2n})$ by allowing a dependence on the three
sets of variables $x$, $y$, and $p$ (cf. \cite{Shubin} Definition 23.3):

\begin{definition}
Let $m\in\mathbb{R}$. The symbol (or amplitude) class $\Pi_{\rho}%
^{m}(\mathbb{R}^{3n})$ consists of all functions $a\in C^{\infty}%
(\mathbb{R}^{3n})$ such that for some $m^{\prime}\in\mathbb{R}$ satisfy
\begin{equation}
|\partial_{p}^{\alpha}\partial_{x}^{\beta}\partial_{y}^{\gamma}a(x,y,p)|\leq
C_{\alpha\beta\gamma}\left\langle u\right\rangle ^{m-\rho|\alpha+\beta
+\gamma|}\left\langle x-y\right\rangle ^{m^{\prime}+\rho|\alpha+\beta+\gamma|}
\label{estim2}%
\end{equation}
for every $(\alpha,\beta,\gamma)\in\mathbb{N}^{3n}$, where $C_{\alpha
\beta\gamma}\geq0$ and $u=(x,y,p)$.
\end{definition}

It turns out that an operator (\ref{aoscille}) with amplitude $b\in\Pi_{\rho
}^{m}(\mathbb{R}^{3n})$ is a Shubin $\tau$-pseudo-differential operator with
symbol in $\Gamma_{\rho}^{m}(\mathbb{R}^{2n})$ -- and this for every value of
the parameter $\tau$:

\begin{proposition}
\label{propimpshu}Let $\tau$ be an arbitrary real number. (i) Every
pseudo-differential operator $\widehat{A}$ of the type (\ref{aoscille}) with
amplitude $b\in\Pi_{\rho}^{m}(\mathbb{R}^{3n})$ can be uniquely written in the
form $\widehat{A}=\operatorname*{Op}_{\tau}(a_{\tau})$ for some symbol
$a_{\tau}\in\Gamma_{\rho}^{m}(\mathbb{R}^{2n})$, that is%
\begin{equation}
\widehat{A}\psi(x)=\left(  \tfrac{1}{2\pi\hbar}\right)  ^{n}\iint e^{\frac
{i}{\hbar}p(x-y)}a_{\tau}((1-\tau)x+\tau y,p)\psi(y)d^{n}yd^{n}p; \label{atau}%
\end{equation}
the symbol $a_{\tau}$ has the asymptotic expansion
\begin{equation}
a_{\tau}(x,p)\sim\sum_{\beta,\gamma}\frac{1}{\beta!\gamma!}\tau^{|\beta
|}(1-\tau)^{|\gamma|}\partial_{p}^{\beta+\gamma}(i\hbar\partial_{x})^{\beta
}(-i\hbar\partial_{y})^{\gamma}b(x,y,p)|_{y=x}. \label{atauasym1}%
\end{equation}
(ii) In particular, choosing $\tau=\frac{1}{2}$, there exists $a_{\mathrm{W}%
}\in\Gamma_{\rho}^{m}(\mathbb{R}^{2n})$ such that $\widehat{A}%
=\operatorname*{Op}_{\mathrm{W}}(a_{\mathrm{W}})$.
\end{proposition}

\begin{proof}
See Shubin \cite{Shubin}, Theorem 23.2 for the case $\hbar=1$ and de Gosson
\cite{Birkbis}, Section 14.2.2.
\end{proof}

We have in addition an asymptotic formula allowing to pass from one $\tau
$-symbol to another when $\widehat{A}$ is given by (\ref{atau}): if
$\widehat{A}=\operatorname*{Op}_{\tau}(a_{\tau})=\operatorname*{Op}%
_{\tau^{\prime}}(a_{\tau^{\prime}})$ with $a_{\tau},a_{\tau^{\prime}}\in
\Pi_{\rho}^{m}(\mathbb{R}^{3n})$, then
\[
a_{\tau}(x,p)\sim\sum_{\alpha\geq0}\frac{i^{-|\alpha|}}{\alpha!}(\tau^{\prime
}-\tau)^{|\alpha|}\partial_{p}^{\alpha}\partial_{x}^{\alpha}a_{\tau^{\prime}%
}(x,p)
\]
(\cite{Shubin}, Theorem 23.3).

\subsection{Elementary properties}

The class of all operators (\ref{aoscille}) with $b\in\Pi_{\rho}%
^{m}(\mathbb{R}^{3n})$ is denoted by $G_{\rho}^{m}(\mathbb{R}^{n})$ (cf.
\cite{Shubin}, Definition 23.4); $G^{-\infty}(\mathbb{R}^{n})=\cap
_{m\in\mathbb{R}}G_{\rho}^{m}(\mathbb{R}^{n})$ consists of all operators
$\mathcal{S}(\mathbb{R}^{n})\longrightarrow\mathcal{S}(\mathbb{R}^{n})$ with
distributional kernel $K\in\mathcal{S}(\mathbb{R}^{n}\times\mathbb{R}^{n})$.
It is useful to make the following remark: in the standard theory of
pseudo-differential operators (notably in its applications to partial
differential operators) it is customary to use operators
\begin{equation}
\widehat{A}\psi(x)=\left(  \tfrac{1}{2\pi}\right)  ^{n}\iint e^{i(x-y)\xi
}a(x,y,\xi)\psi(y)d^{n}yd^{n}\xi\label{aa}%
\end{equation}
which correspond, replacing $p$ with $\xi$ to the choice $\hbar=1$ in the
expression (\ref{aoscille}). It is in fact easy to toggle between the
expression above and its $\hbar$-dependent version: one just replaces
$a(x,y,\xi)$ with $a(x,y,p)$ and $\xi$ with $p/\hbar$ so that $d\xi=\hbar
^{-n}dp.$ However, when doing this, one must be careful to check that the
amplitude $a(x,y,\xi)$ and $a(x,y,\hbar\xi)$ belong to the same symbol class.
That this is indeed always the case when one deals with Shubin classes is
clear from Lemma \ref{Lemma1}. The following situation is important in our
context; consider the $\hbar=1$ Weyl operator%
\begin{equation}
\widehat{A}\psi(x)=\left(  \tfrac{1}{2\pi}\right)  ^{n}\iint e^{i(x-y)\xi
}a(\tfrac{1}{2}(x+y),\xi)\psi(y)d^{n}yd^{n}\xi. \label{aw}%
\end{equation}
Denoting by $\widehat{A}^{(\hbar)}$ the corresponding operator (\ref{aoscille}%
) in order to make the $\hbar$-dependence clear, that is%
\begin{equation}
\widehat{A}^{(\hbar)}\psi(x)=\left(  \tfrac{1}{2\pi\hbar}\right)  ^{n}\iint
e^{\frac{i}{\hbar}p(x-y)}a(\tfrac{1}{2}(x+y),p)\psi(y)d^{n}yd^{n}p, \label{ah}%
\end{equation}
we have $\widehat{A}^{(\hbar)}=\widehat{M}_{\hbar}^{-1}\widehat{B}\widehat
{M}_{\hbar}$ where $\widehat{B}$ is the operator (\ref{aw}) with symbol
$b(x,p)=a(\hbar^{1/2}x,\hbar^{1/2}p)$ and $\widehat{M}_{\hbar}$ is the unitary
scaling operator defined by $\widehat{M}_{\hbar}\psi(x)=\hbar^{n/4}\psi
(\hbar^{1/2})$.

Using the symbol estimates (\ref{est1}) it is straightforward to show that
every operator $\widehat{A}\in G_{\rho}^{m}(\mathbb{R}^{n})$ is a continuous
operator $\mathcal{S}(\mathbb{R}^{n})\longrightarrow\mathcal{S}(\mathbb{R}%
^{n})$ and can hence be extended into a continuous operator $\mathcal{S}%
^{\prime}(\mathbb{R}^{n})\longrightarrow\mathcal{S}^{\prime}(\mathbb{R}^{n})$.
It follows by duality that if $\widehat{A}\in G_{\rho}^{m}(\mathbb{R}^{n})$,
then $\widehat{A}^{\ast}\in G_{\rho}^{m}(\mathbb{R}^{n})$ (cf. \cite{Shubin},
Theorem 23.5).

One also shows that (\cite[Theorem 23.6]{Shubin}) if $\widehat{A}\in G_{\rho
}^{m}(\mathbb{R}^{n})$ and $\widehat{B}\in G_{\rho}^{m^{\prime}}%
(\mathbb{R}^{n})$ then $\widehat{C}=\widehat{A}\widehat{B}\in G_{\rho
}^{m+m^{\prime}}(\mathbb{R}^{n})$.

\section{Weyl versus Born-Jordan symbol }

\subsection{General results}

Comparing the expressions (\ref{AW1}) and (\ref{abjharmonic}) giving the
harmonic representations of respectively Weyl and Born-Jordan operators one
sees that if $\widehat{A}=\operatorname*{Op}_{\mathrm{W}}%
(a)=\operatorname*{Op}_{\mathrm{BJ}}(b)$ then the symbols $a$ and $b$ are
related by the convolution relation $b\ast\theta_{\mathrm{BJ}}=a$;
equivalently, taking the (symplectic) Fourier transform of each side
\begin{equation}
a_{\sigma}(z)=b_{\sigma}(z)\operatorname{sinc}\left(  \frac{px}{2\hbar
}\right)  .\label{bsigasig}%
\end{equation}
The difficulty in recovering $b_\sigma$ from $a_\sigma$ comes from the fact that the $\operatorname{sinc}$ function has
infinitely many zeroes; in fact $\operatorname{sinc}(px/2\hbar)=0$ for all
points $z=(x,p)$ such that $px=2N\pi\hbar$ for a non-zero integer $N$. We are
thus confronted with a division problem. Notice in addition that if the
solution $b$ exists then it is not unique: assume that $c(z)=e^{-i\sigma
(z,z_{0})/\hbar}$ where $p_{0}x_{0}=2N\pi\hbar$ ($N\in\mathbb{Z}$, $N\neq0$).
We have $c_{\sigma}(z)=(2\pi\hbar)^{n}\delta(z-z_{0})$ and hence by
(\ref{abjharmonic})
\[
\operatorname*{Op}\nolimits_{\mathrm{BJ}}(c)=\int\delta(z-z_{0}%
)\operatorname{sinc}\left(  \tfrac{px}{2\hbar}\right)  \widehat{T}%
(z)d^{2n}z=0.
\]
\ \ It follows that if $\operatorname*{Op}_{\mathrm{BJ}}(b)=\operatorname*{Op}%
_{\mathrm{W}}(a)$ then we also have $\operatorname*{Op}_{\mathrm{BJ}%
}(b+c)=\operatorname*{Op}_{\mathrm{W}}(a)$. Now, in \cite[Theorem 7]{cogoni15}
we have proven that the equation (\ref{bsigasig}) always has a (non-unique)
solution in $b\in\mathcal{S}(\mathbb{R}^{2n})$ for every given $a\in \mathcal{S}(\mathbb{R}^{2n})$; our proof used the theory of division of distributions. Thus every
Weyl operator has a Born-Jordan symbol; equivalently

\begin{proposition}
For every continuous linear operator $\widehat{A}:\mathcal{S}(\mathbb{R}%
^{n})\longrightarrow\mathcal{S}^{\prime}(\mathbb{R}^{n})$ there exists
$b\in\mathcal{S}^{\prime}(\mathbb{R}^{2n})$ such that $\widehat{A}%
=\operatorname*{Op}_{\mathrm{BJ}}(b)$.
\end{proposition}

Notice that the existence of the solution $b$ of (\ref{bsigasig}), as
established in \cite{cogoni15}, is a purely qualitative result; it does
not tell us anything on the properties of that solution.

\subsection{Weyl symbol of a Born-Jordan operator}

We are going to show that every Born-Jordan operator with symbol in one of
the Shubin classes $\Gamma_{\rho}^{m}(\mathbb{R}^{2n})$ is a Weyl operator
with symbol {\it in the same symbol class} and produce an asymptotic expansion for
the latter. For this we will need the following elementary inequalities:

\begin{lemma}
Let $\xi$ and $\eta$ be positive numbers and $m\in\mathbb{R}$. We have
\begin{equation}
\min\{\xi^{m},\eta^{m}\}\leq C(\xi+\eta)^{m}\label{miniksi}%
\end{equation}
where $C=\max\{1,2^{-m}\}$ and
\begin{equation}
(1+|\xi-\eta|^{2})^{m}\leq2^{|m|}(1+|\xi|^{2})^{m}(1+|\eta|^{2})^{|m|}%
\label{Peetre}%
\end{equation}

\end{lemma}

\begin{proof}
Proof of (\ref{miniksi}). The case $m\geq0$ is straightforward: we have
\[
\min\{\xi^{m},\eta^{m}\}\leq\xi^{m}+\eta^{m}\leq(\xi+\eta)^{m}.
\]
Suppose $m<0$; if $\xi\leq\eta$ we have
\[
\min\{\xi^{m},\eta^{m}\}=\eta^{m}\leq2^{-m}(\xi+\eta)^{m};
\]
the case $\xi>\eta$ follows in the same way. Proof of (\ref{Peetre}): see for
instance Chazarain and Piriou \cite{chapi7} or H\"{o}rmander \cite{69}.
\end{proof}

The estimate (\ref{Peetre}) is usually referred to as Peetre's inequality in
the literature on pseudo-differential operators.

\begin{theorem}
\label{corollashu} Let $\widehat{A}_{\mathrm{BJ}}=\operatorname*{Op}%
_{\mathrm{BJ}}(a)$ with symbol $a\in\Gamma_{\rho}^{m}(\mathbb{R}^{2n})$. (i)
For every $\tau\in\mathbb{R}$ there exists $a_{\tau}\in\Gamma_{\rho}%
^{m}(\mathbb{R}^{2n})$ such that $\widehat{A}_{\mathrm{BJ}}=\operatorname*{Op}%
_{\tau}(a_{\tau})$. Here $a_{\tau}$ has the following asymptotic expansion
\begin{equation}
a_{\tau}(x,p)\sim\sum_{\alpha\in\mathbb{N}^{n}}\frac{(i\hbar)^{|\alpha|}%
(\tau^{|\alpha|+1}-(\tau-1)^{|\alpha|+1})}{\alpha!(|\alpha|+1)}\partial
_{x}^{\alpha}\partial_{p}^{\alpha}a(x,p). \label{atauabj}%
\end{equation}
(ii) In particular $\widehat{A}_{\mathrm{BJ}}=\operatorname*{Op}_{\mathrm{BJ}%
}(a)$ is a Weyl operator $\widehat{A}_{\mathrm{W}}=\operatorname*{Op}%
_{\mathrm{W}}(a_{\mathrm{W}})$ with symbol $a_{\mathrm{W}}\in\Gamma_{\rho}%
^{m}(\mathbb{R}^{2n})$, having the asymptotic expansion
\begin{equation}
a_{\mathrm{W}}(x,p)\sim\sum_{%
\genfrac{}{}{0pt}{}{\alpha\in\mathbb{N}^{n}}{|\alpha|\,{\mathrm{even}}}%
}\frac{1}{\alpha!(|\alpha|+1)}\left(\frac{i\hbar}{2}\right)^{|\alpha|}\partial
_{x}^{\alpha}\partial_{p}^{\alpha}a(x,p) \label{aweylbj}%
\end{equation}
and we have $a_{\mathrm{W}}-a\in\Gamma_{\rho}^{m-2\rho}(\mathbb{R}^{2n})$.
\end{theorem}

\begin{proof}
Property (ii) follows from (i) choosing $\tau=\frac{1}{2}$. (ii) Consider the
$\tau$-pseudo-differential operator $\widehat{A}_{\tau}=\operatorname*{Op}%
_{\tau}(a)$:
\[
\widehat{A}_{\tau}\psi(x)=\left(  \tfrac{1}{2\pi\hbar}\right)  ^{n}\int
e^{\frac{i}{\hbar}p(x-y)}a((1-\tau)x+\tau y,p)\psi(y)d^{n}yd^{n}p
\]
and set
\begin{equation}
a_{\mathrm{BJ}}(x,y,p)=\int_{0}^{1}a((1-\tau)x+\tau y,p)d\tau. \label{abj7}%
\end{equation}
We thus have, using \eqref{2.10bis},
\[
\widehat{A}_{\mathrm{BJ}}\psi(x)=\left(  \tfrac{1}{2\pi\hbar}\right)  ^{n}\int
e^{\frac{i}{\hbar}p(x-y)}a_{\mathrm{BJ}}(x,y,p)\psi(y)d^{n}yd^{n}p
\]
which is of the type (\ref{aoscille}). Let us show that $a_{\mathrm{BJ}}\in
\Pi_{\rho}^{m}(\mathbb{R}^{3n})$, \textit{i.e.}\ that we have estimates of the
type
\begin{equation}
|\partial_{x}^{\alpha}\partial_{y}^{\beta}\partial_{p}^{\gamma}a_{\mathrm{BJ}%
}(x,y,p)|\leq C_{\alpha,\beta,\gamma}\langle(x,y,p)\rangle^{m-\rho
|\alpha+\beta+\gamma|}\langle x-y\rangle^{m^{\prime}+\rho|\alpha+\beta
+\gamma|} \label{babaorum}%
\end{equation}
for some $m^{\prime}\in\mathbb{R}$ independent of $\alpha,\beta,\gamma$. The
result will follow using Proposition \ref{propimpshu}. Let us set
\[
b_{\tau}(x,y,p)=a((1-\tau)x+\tau y,p);
\]
we have
\[
\partial_{x}^{\alpha}\partial_{y}^{\beta}\partial_{p}^{\gamma}b_{\tau
}(x,y,p)=(1-\tau)^{|\alpha|}\tau^{|\beta|}(\partial_{x}^{\alpha+\beta}%
\partial_{p}^{\gamma}a)((1-\tau)x+\tau y,p)
\]
hence, since $a\in\Gamma_{\rho}^{m}(\mathbb{R}^{n})$, we have by \eqref{est1}
the estimates
\begin{equation}
|\partial_{x}^{\alpha}\partial_{y}^{\beta}\partial_{p}^{\gamma}b_{\tau
}(x,y,p)|\leq C_{\alpha+\beta,\gamma}(1-\tau)^{|\alpha|}\tau^{|\beta|}%
\langle((1-\tau)x+\tau y,p)\rangle^{m-\rho|\alpha+\beta+\gamma|}.
\label{babaorum2}%
\end{equation}
Now, by Peetre's inequality (\ref{Peetre}) there exists a constant
$C=C(m,\rho,\alpha,\beta,\gamma)>0$ such that the estimates
\[
\langle((1-\tau)x+\tau y,p)\rangle^{m-\rho|\alpha+\beta+\gamma|}\leq
C\langle(x,p)\rangle^{m-\rho|\alpha+\beta+\gamma|}\langle\tau(x-y)\rangle
^{|m|+\rho|\alpha+\beta+\gamma|}%
\]
and
\[
\langle((1-\tau)x+\tau y,p)\rangle^{m-\rho|\alpha+\beta+\gamma|}\!\leq
C\langle(y,p)\rangle^{m-\rho|\alpha+\beta+\gamma|}\langle(1-\tau
)(x-y)\rangle^{|m|+\rho|\alpha+\beta+\gamma|}%
\]
hold, hence
\begin{multline*}
\langle((1-\tau)x+\tau y,p)\rangle^{m-\rho|\alpha+\beta+\gamma|}\leq\\
C\min\{\langle(x,p)\rangle^{m-\rho|\alpha+\beta+\gamma|},\langle
(y,p)\rangle^{m-\rho|\alpha+\beta+\gamma|}\}\langle x-y\rangle^{|m|+\rho
|\alpha+\beta+\gamma|}.
\end{multline*}
This implies, using the inequality (\ref{miniksi}), that
\[
\langle((1-\tau)x+\tau y,p)\rangle^{m-\rho|\alpha+\beta+\gamma|}\leq
C^{\prime}\langle(x,y,p)\rangle^{m-\rho|\alpha+\beta+\gamma|}\langle
x-y\rangle^{|m|+\rho|\alpha+\beta+\gamma|}.
\]
Together with (\ref{babaorum2}) this inequality implies (\ref{babaorum}) with
$m^{\prime}=|m|$ after an integration on $\tau$. The asymptotic expansion
\eqref{atauabj} follows by using the expansion of the $\tau$-symbol $a_{\tau}$
in \eqref{atauasym1}, in terms of the amplitude $b(x,y,p)=a_{\mathrm{BJ}%
}(x,y,p)$ in \eqref{abj7}. Namely, observe that
\begin{align*}
\partial_{p}^{\beta+\gamma}(i\hbar\partial_{x})^{\beta}(-i\hbar\partial
_{y})^{\gamma}  &  b(x,y,p)|_{y=x}\\
&  =(i\hbar)^{|\beta+\gamma|}(-1)^{|\gamma|}\int_{0}^{1}(1-t)^{|\beta
|}t^{|\gamma|}\partial_{p}^{\beta+\gamma}\partial_{x}^{\beta+\gamma
}a(x,p)\,dt\\
&  =(i\hbar)^{|\beta+\gamma|}(-1)^{|\gamma|}\partial_{p}^{\beta+\gamma
}\partial_{x}^{\beta+\gamma}a(x,p)\int_{0}^{1}(1-t)^{|\beta|}t^{|\gamma|}dt.
\end{align*}
Setting $\alpha=\beta+\gamma$, so that $\beta=\alpha-\gamma$, we have
\[
\frac{1}{\beta!\gamma!}={\binom{\alpha}{\gamma}}\frac{1}{\alpha!}%
\]
and hence
\begin{align*}
\sum_{\gamma\leq\alpha}{\binom{\alpha}{\gamma}}[\tau(1-t)]^{|\alpha-\gamma
|}(-1)^{|\gamma|}  &  [(1-\tau)t]^{|\gamma|}\\
&  =\prod_{j=1}^{n}\sum_{\gamma_{j}\leq\alpha_{j}}{\binom{\alpha_{j}}%
{\gamma_{j}}}[\tau(1-t)]^{|\alpha_{j}-\gamma_{j}|}[-(1-\tau)t]^{|\gamma_{j}%
|}\\
&  =\prod_{j=1}^{n}[\tau(1-t)-(1-\tau)t]^{\alpha_{j}}\\
&  =(\tau-t)^{|\alpha|}.
\end{align*}
Computing the integral
\[
\int_{0}^{1}(\tau-t)^{|\alpha|}\,dt=\frac{\tau^{|\alpha|+1}-(\tau
-1)^{|\alpha|+1}}{|\alpha|+1}%
\]
we immediately obtain the asymptotic expansion for $a_{\tau}(x,p)$ in
\eqref{atauabj}. This concludes the proof.
\end{proof}

Notice that the asymptotic formula (\ref{aweylbj}) yields \textit{exact}
results when the Born-Jordan symbol $a$ is a polynomial in the variables
$x_{j},p_{k}$. For instance, when $n=1$ and $a(z)=a_{rs}(z)=x^{r}p^{s}$ it
leads to%
\begin{equation}
a_{rs,\mathrm{W}}(x,p)=\sum_{\substack{k\leq\inf(r,s)\\k\text{ \textrm{even}}%
}}\left(  \frac{i\hbar}{2}\right)  ^{k}\frac{k!}{k+1}\binom{r}{k}\binom{s}%
{k}x^{r-k}p^{s-k}.\label{ars}%
\end{equation}
We refer to Domingo and Galapon \cite{34} for a general discussion of
quantization of monomials.

Using \cite{Shubin}, Definition 23.4, the result above has the following
interesting consequence:

\begin{corollary}
\label{e1} A Born-Jordan operator $\widehat{A}_{\mathrm{BJ}}%
=\operatorname*{Op}_{\mathrm{BJ}}(a)$ with symbol $a\in\Gamma_{\rho}%
^{m}(\mathbb{R}^{2n})$ belongs to $G_{\rho}^{m}(\mathbb{R}^{n})$.
\end{corollary}

This result reduces in many cases the study of Born-Jordan operators to that
of Shubin operators.

\subsection{The Born-Jordan symbol of a Weyl operator}

We now address the more difficult problem of finding the Born-Jordan symbol
of a given Weyl operator in $G_{\rho}^{m}(\mathbb{R}^{2n})$. As already
observed the analysis in \cite{cogoni15} did not provide an explicit formula
for it because of division problems. It is remarkable that, nevertheless, an
explicit and general \textit{asymptotic expansion} can be written down when
the symbol belongs to one of the classes $\Gamma_{\rho}^{m}(\mathbb{R}^{2n})$.
To this end we need a preliminary lemma about the formal power series arising
in \eqref{aweylbj}.

\begin{lemma}
\label{lemma3} Consider the power series
\[
\sum_{%
\genfrac{}{}{0pt}{}{\alpha\in\mathbb{N}^{n}}{|\alpha|\ \mathrm{even}}%
}\frac{1}{\alpha!(|\alpha|+1)}x^{\alpha}.
\]
Its formal reciprocal is given by the series $\sum_{\alpha\in\mathbb{N}^{n}%
}\frac{c_{\alpha}}{\alpha!}x^{\alpha}$ where $c_{0}=1$ and, for $\alpha
\not =0$,
\begin{equation}
c_{\alpha}=\alpha!\sum_{j=1}^{|\alpha|}(-1)^{j}\sum_{%
\genfrac{}{}{0pt}{}{\alpha^{(1)}+\ldots+\alpha^{(j)}=\alpha}{|\alpha
^{(1)}|,\ldots,|\alpha^{(j)}|\not =0\ \mathrm{even}}%
}\frac{1}{\alpha^{(1)}!\cdots\alpha^{(j)}!(|\alpha^{(1)}|+1)\cdots(|\alpha^{(j)}|+1)}.\label{calpha}%
\end{equation}

\end{lemma}

\begin{proof}
The proof is straightforward: we expand
\[
\Big(1+\sum_{|\alpha|\not =0\ \mathrm{even}}\frac{1}{\alpha!(|\alpha|+1)}x^{\alpha
}\Big)^{-1}%
\]
as a geometric series and collect the similar terms. Alternatively, we could
also apply the Fa\`{a} di Bruno formula generalizing the chain rule to the
derivatives at $x=0$ of the function $x\longmapsto g(f(x))$ where
$g(t)=1/(1+t)$ and $f(x)=\sum\limits_{|\alpha|\not =0\ \mathrm{even}}\frac
{1}{\alpha!(|\alpha|+1)}x^{\alpha}$.
\end{proof}

Let us now prove our second main result:

\begin{theorem}
Consider a Weyl operator $\widehat{A}_{\mathrm{W}}=\operatorname*{Op}%
_{\mathrm{W}}(a)$ with Weyl symbol $a\in\Gamma_{\rho}^{m}(\mathbb{R}^{2n})$.
Let $b\in\Gamma_{\rho}^{m}(\mathbb{R}^{2n})$ be any symbol (whose existence is
guaranteed by Proposition \ref{23.1.}) with the following asymptotic
expansion:
\begin{equation}
b(x,p)\sim\sum_{%
\genfrac{}{}{0pt}{}{\alpha\in\mathbb{N}^{n}}{|\alpha|\,{\mathrm{even}}}%
}\frac{c_{\alpha}}{\alpha!}\left(
\frac{i\hbar}{2}\right)  ^{|\alpha|}\partial_{x}^{\alpha}\partial_{p}^{\alpha
}a(x,p),\label{ases}%
\end{equation}
where the coefficients $c_{\alpha}$ are given in \eqref{calpha} ($c_{0}=1$).

Let $\widehat{A}_{\mathrm{BJ}}=\operatorname*{Op}_{\mathrm{BJ}}(b)$ be the
corresponding Born-Jordan operator. Then
\begin{equation}
\widehat{A}_{\mathrm{BJ}}=\widehat{A}_{\mathrm{W}}+R\label{aar}%
\end{equation}
where $R$ is a pseudodifferential operator with symbol in the Schwartz space
$\mathcal{S}(\mathbb{R}^{2n})$.

\begin{proof}
The operator $\widehat{A}_{\mathrm{BJ}}=\operatorname*{Op}_{\mathrm{BJ}}(b)$
by Theorem \ref{corollashu} can be written as a Weyl operator with Weyl
symbol
\begin{equation}
a_{\mathrm{W}}(x,p)\sim\sum_{%
\genfrac{}{}{0pt}{}{\alpha\in\mathbb{N}^{n}}{|\alpha|\,{\mathrm{even}}}%
}\frac{1}{\alpha!(|\alpha|+1)}\left(  \frac{i\hbar}{2}\right)  ^{|\alpha
|}\partial_{x}^{\alpha}\partial_{p}^{\alpha}b(x,p).\label{awabj}%
\end{equation}
Now we substitute in this expression the asymptotic expansion \eqref{ases} for
$b$ and we use the fact that the formal differential operators given by the
series
\[
\sum_{%
\genfrac{}{}{0pt}{}{\alpha\in\mathbb{N}^{n}}{|\alpha|\,{\mathrm{even}}}%
}\frac{1}{\alpha!(|\alpha|+1)}\left(  \frac{i\hbar}{2}\right)^{|\alpha|}\partial
_{x}^{\alpha}\partial_{p}^{\alpha}\text{ \ \textit{and} }\sum_{%
\genfrac{}{}{0pt}{}{\alpha\in\mathbb{N}^{n}}{|\alpha|\,{\mathrm{even}}}%
}\frac{c_{\alpha}}{\alpha!}\left(  \frac{i\hbar}{2}\right)
^{|\alpha|}\partial_{x}^{\alpha}\partial_{p}^{\alpha}%
\]
are inverses of each other in view of Lemma \ref{lemma3} (to see this,
formally replace $x=(x_{1},\ldots,x_{n})$ in Lemma \ref{lemma3} by
$(i\hbar/2)(\partial_{x_{1}}\partial_{p_{1}},\ldots,\partial_{x_{n}}%
\partial_{p_{n}}$). It follows that
\begin{equation}
a_{\mathrm{W}}-a\in\bigcap_{N\in\mathbb{N}}\Gamma_{\rho}^{-N}(\mathbb{R}%
^{2n})=\mathcal{S}(\mathbb{R}^{2n})\label{awas}%
\end{equation}
hence (\ref{aar}).
\end{proof}
\end{theorem}

Notice that in dimension $n=1$ we have
\[
\sum_{k\geq0\ \mathrm{even}}\frac{1}{k!(k+1)}x^{k}=\frac{1}{x}\sum_{k\geq0\ \mathrm{even}}\frac{1}{(k+1)!}x^{k+1} =\frac{\sinh x}{x} =:F(x)
\]
so that $c_{k}=\partial_{x}^{k}(1/F(x))|_{x=0}$. In particular, $c_{k}=0$ for
odd $k$. In this case the series expansion of $1/F(x)$ is particularly easy, since it
coincides with the MacLaurin series expansion of the function 
$$\frac{1}{F(x)}= x \, \mbox{cosech} \,x = \sum_{k\geq0\ \mathrm{even}}\frac{2-2^k}{k!} B_k x^{k}
$$
where the $B_k$ are the Bernoulli numbers
\begin{equation}\label{Bernoulli}
B_k =\lim_{x\rightarrow 0} f^{(k)}(x),
\end{equation} 
with $$ f(x)= \frac{x}{e^x-1}.$$
More explicitly,
$$\frac{1}{F(x)}= x \, \mbox{cosech} \,x = 1-\frac16 x^2+\frac{7}{360}x^4-\frac{31}{15120}x^6+\frac{127}{604800}x^8-\cdots
$$
and the coefficients $c_k$, with $k$ even, are provided by
\[
c_{0}=1,\ c_{2}=-\frac{1}{3},\ c_{4}=\frac{7}{15},\ c_{6}=-\frac{31}{21},\ c_{8}=\frac{127}{15},\ \ldots\quad(n=1).
\]
In this case formula (\ref{ases}) takes the simple form
\[
b(x,p)\sim\sum_{k\geq0\text{ \textrm{even}}}\frac{c_{k}}{k!}\left(
\frac{i\hbar}{2}\right)  ^{k}\partial_{x}^{k}\partial_{p}^{k}a(x,p).
\]
As in the case of formula (\ref{ars}), the asymptotic expansion (\ref{ases})
becomes exact (and reduces to a finite sum) when the symbol $a$ is a
polynomial. For instance, assuming $n=1$ choose $a(z)=a_{rs}(z)=x^{r}p^{s}$.
Then the formula above yields
\begin{align}
b_{rs,\mathrm{BJ}}(x,p)&=\sum_{k\leq\inf(r,s)\text{ }k\text{ \textrm{even}}%
}k!  c_{k}\left(  \frac{i\hbar}{2}\right)^{k} \binom{r}{k}\binom{s}{k}%
x^{r-k}p^{s-k}.\label{arsbis}\\
&=\sum_{k\leq\inf(r,s)\text{ }k\text{ \textrm{even}}%
} k! (2-2^k) B_k\left(  \frac{i\hbar}{2}\right)^{k} \binom{r}{k}\binom{s}{k}%
x^{r-k}p^{s-k}.\notag
\end{align}
where the $B_k$ are the Bernoulli numbers defined in \eqref{Bernoulli}.

We also make the following remark: formulas (\ref{awabj}) and (\ref{awas})
show that (modulo a term in $\mathcal{S}(\mathbb{R}^{2n})$) a Weyl operator
with symbol in $\Gamma_{\rho}^{m}(\mathbb{R}^{2n})$ has a Born-Jordan symbol
belonging to the same class $\Gamma_{\rho}^{m}(\mathbb{R}^{2n})$. This is
however by no means a uniqueness result since, as we have already observed, we have $\operatorname*{Op}%
_{\mathrm{BJ}}(b+c)=0$ for all symbols $c(z)=e^{-i\sigma(z,z_{0})/\hbar}$
where $p_{0}x_{0}=2N\pi\hbar$ ($N\in\mathbb{Z}$, $N\neq0$). Observe that such a symbol
$c$ belongs to none of the symbol classes $\Gamma_{\rho}^{m}(\mathbb{R}^{2n})$.

\section{Regularity and Global Hypoellipticity Results}

In order to define the Sobolev--Shubin spaces (cf.\
\cite{Shubin}, Definition 25.3), we recall the definition of  anti-Wick operators.  The anti-Wick operator $\operatorname*{Op}_{\mathrm{AW}}(a)$ with symbol $a$ is defined by
\[\mathrm{Op}_{\mathrm{AW}}(a)f=\int a(z) P_z f d^{2n} z,
\]
where $P_z f(t)= \langle f, \Phi_z\rangle \Phi_z(t)$ are orthogonal projections on $L^2(\mathbb{R}^{n})$ on the functions $$\Phi_z(t)=\pi^{-n/4} e^{i tp } e^{-\frac{|t-x|^2}{2}},\quad z=(x,p)\in  \mathbb{R}^{2n}$$ (i.e., phase-space shifts of the Gaussian $\pi^{-n/4}e^{-\frac{|t|^2}{2}})$.
\begin{definition}
For $s\in\mathbb{R}$  consider the anti-Wick symbol $\langle z\rangle^s$, $z\in \mathbb{R}^{2n}$, and let $A_s=\operatorname*{Op}_{\mathrm{AW}}(a)$ be the corresponding anti-Wick operator.  The Sobolev--Shubin space $Q^{s}$ is defined by
\[
Q^{s}=\{ f\in S'(\mathbb{R}^{n})\,:  A_s f\in L^2(\mathbb{R}^{n}) \}  =A_s^{-1}L^2(\mathbb{R}^{n}).
\]
\end{definition}

We then have the following continuity result for Born-Jordan operators:

\begin{proposition}
Let $\widehat{A}_{\mathrm{BJ}}=\operatorname*{Op}_{\mathrm{BJ}}(a)$ with
symbol $a\in\Gamma_{\rho}^{m}(\mathbb{R}^{2n})$. We have
\[
\widehat{A}_{\mathrm{BJ}}\,:Q^{s}(\mathbb{R}^{n})\rightarrow Q^{s-m}%
(\mathbb{R}^{n}).
\]

\end{proposition}

\begin{proof}
By Corollary \eqref{e1} the operator $\widehat{A}_{\mathrm{BJ}}$ is in the
class $G_{\rho}^{m}(\mathbb{R}^{n})$. The result follows by applying Theorem
25.2 in \cite{Shubin}.
\end{proof}

It turns out that the Sobolev--Shubin spaces are particular cases of Feichtinger's modulation spaces
\cite{37,38,65}; we do not discuss these here and refer to Cordero et al
\cite{cogoni16b} for a study of continuity properties of Born-Jordan
operators in these spaces. In fact, our reduction result Theorem
\ref{corollashu} allows to transpose to Born-Jordan operators all known
regularity results for Weyl operators with symbol in the symbol classes
$\Gamma_{\rho}^{m}(\mathbb{R}^{2n})$. For instance:

\begin{proposition}
Let $a\in\Gamma_{\rho}^{0}(\mathbb{R}^{2n})$. Then $\widehat{A}_{\mathrm{BJ}%
}=\operatorname*{Op}_{\mathrm{BJ}}(a)$ is a bounded operator on $L^{2}%
(\mathbb{R}^{n})$.
\end{proposition}

\begin{proof}
In view of Theorem 24.3 in \cite{Shubin} every Weyl operator with symbol in
$\Gamma_{\rho}^{0}(\mathbb{R}^{2n})$ is bounded on $L^{2}(\mathbb{R}^{n})$;
the result follows in view of Theorem \ref{corollashu} (ii).
\end{proof}

We next recall the notion of global hypoellipticity \cite{Shubin,Treves},
which plays an important role in the study of spectral theory for
pseudo-differential operators (see the monograph \cite{boburo96} by Boggiatto
\textit{et al}.). An operator $\widehat{A}:\mathcal{S}^{\prime} (\mathbb{R}%
^{n})\longrightarrow\mathcal{S}^{\prime}(\mathbb{R}^{n})$ which also maps $\mathcal{S}(\mathbb{R}%
^{n})$ into itself is \textit{globally
hypoelliptic} if
\[
\psi\in\mathcal{S}^{\prime}(\mathbb{R}^{n})\text{ and }\widehat{A}\psi
\in\mathcal{S}(\mathbb{R}^{n})\Longrightarrow\psi\in\mathcal{S}(\mathbb{R}%
^{n})\text{.}%
\]
(global hypoellipticity is thus not directly related to usual notion of
hypoellipticity \cite{69}, which is a local notion).

In \cite{Shubin} Shubin introduces the following subclass of $\Gamma_{\rho
}^{m}(\mathbb{R}^{2n})$:

\begin{definition}
Let $m,m_{0}\in\mathbb{R}$ and $0<\rho\leq1$. The symbol class $H\Gamma_{\rho
}^{m,m_{0}}(\mathbb{R}^{2n})$ consists of all complex functions $a\in
C^{\infty}(\mathbb{R}^{2n})$ such that
\[
C_{0}\left\langle z\right\rangle ^{m_{0}}\leq|a(z)|\leq C_{1}\left\langle
z\right\rangle ^{m}\quad {\rm for}\ |z|>R,%
\]
for some $C_0,C_1,R>0$ and whose derivatives satisfy the following property: for every $\alpha
\in\mathbb{N}^{2n}$ there exists $C_{\alpha}>0$ such that%
\[
|\partial_{z}^{\alpha}a(z)|\leq C_{\alpha}|a(z)|\left\langle z\right\rangle
^{-\rho|\alpha|} \quad {\rm for}\ |z|>R.
\]

\end{definition}

The symbol class $H\Gamma_{\rho}^{m,m_{0}}(\mathbb{R}^{2n})$ is insensitive to
perturbations by lower order terms (\cite{Shubin}, Lemma 25.1, (c)):

\begin{lemma}
\label{Lemmalower}Let $a\in H\Gamma_{\rho}^{m,m_{0}}(\mathbb{R}^{2n})$ and
$b\in\Gamma_{\rho}^{m^{\prime}}(\mathbb{R}^{2n})$. If $m^{\prime}<m_0$ then
$a+b\in H\Gamma_{\rho}^{m,m_{0}}(\mathbb{R}^{2n}).$
\end{lemma}

The interest of these symbol classes comes from the following property
\cite{Shubin}: if $a\in H\Gamma_{\rho}^{m,m_{0}}(\mathbb{R}^{2n})$ then the
Weyl operator $\widehat{A}=\operatorname*{Op}_{\mathrm{W}}(a)$ is globally
hypoelliptic. For instance, the Hermite operator $-\Delta+|x|^{2}$ is globally
hypoelliptic since its Weyl symbol is $a(z)=|z|^{2}$, which is in $H\Gamma
_{1}^{2,2}(\mathbb{R}^{2n})$. Moreover, if $\widehat{A}=\operatorname*{Op}_{\mathrm{W}}(a)$ with $a\in H\Gamma_{\rho}^{m,m_{0}}(\mathbb{R}^{2n})$ the following stronger result holds:%
\[
\psi\in\mathcal{S}^{\prime}(\mathbb{R}^{n})\text{ and }\widehat{A}\psi\in
Q^{s}(\mathbb{R}^{n})\Longrightarrow\psi\in Q^{s+m_0}(\mathbb{R}^{n})\text{.}%
\]

\begin{proposition}
Let $a\in H\Gamma_{\rho}^{m,m_{0}}(\mathbb{R}^{2n})$, with $m-2\rho<m_0$. (i) The Born-Jordan
operator $\widehat{A}_{\mathrm{BJ}}=\operatorname*{Op}_{\mathrm{BJ}}(a)$ is
globally hypoelliptic; (ii) if $\psi$ is a tempered distribution such that
$\widehat{A}_{\mathrm{BJ}}\psi\in Q^{s}(\mathbb{R}^{n})$ for some $s\in\mathbb{R}$ then
$\psi\in Q^{s+m_0}(\mathbb{R}^{n})$.
\end{proposition}

\begin{proof}
In view of the discussion above it suffices to show that the Weyl symbol
$a_{\mathrm{W}}$ of $\widehat{A}_{\mathrm{BJ}}$ belongs to the class
$H\Gamma_{\rho}^{m,m_{0}}(\mathbb{R}^{2n})$. Now, by Theorem \ref{corollashu}
(ii) we have $a_{\mathrm{W}}-a\in\Gamma_{\rho}^{m-2\rho}(\mathbb{R}^{2n})$
hence the result follows using of Lemma \ref{Lemmalower}.
\end{proof}



\begin{thebibliography}{99}                                                                                               %

\bibitem {bogetal} P. Boggiatto, G. De Donno, A. Oliaro, Time-Frequency
Representations of Wigner Type and Pseudo-Differential Operators, Trans. Amer.
Math. Soc. 362(9) (2010) 4955--4981.

\bibitem {bodeol2013} P. Boggiatto, G. De Donno, A. Oliaro, Hudson's theorem for {$\tau$}-{W}igner transforms, Bull. Lond. Math. Soc. 45(6) (2013) 1131--1147.

\bibitem {boburo96} P. Boggiatto, E. Buzano, L. Rodino, Global Hypoellipticity
and Spectral Theory, Akademie Verlag, Math. Research 92, 1996.

\bibitem{bonet} Jos\'e Bonet, L. Frerick, Enrique Jord\`ac, The division problem for tempered distributions of one variable, J. Funct. Anal. 262(5) (2012) 2349--2358.

\bibitem {bj}M. Born, P. Jordan, Zur Quantenmechanik, Z. Physik 34 (1925) 858--888.

\bibitem {bjh}M. Born, W. Heisenberg, P. Jordan, Zur Quantenmechanik II, Z.
Physik 35 (1925) 557--615.

\bibitem {chapi7}J. Chazarain, A. Piriou, Introduction \`{a} la th\'{e}orie
des \'{e}quations aux d\'{e}riv\'{e}es partielles lin\'{e}aires,
Gauthier--Villars, Paris, (1981). English translation: Introduction to the
theory of linear partial differential equations, Studies in Mathematics and
its Applications, 14, North--Holland, 1982.

\bibitem {Cohen1}L. Cohen, Generalized phase-space distribution functions. J.
Math. Phys. 7 (1966) 781--786.

\bibitem {Cohenbook}L. Cohen, The Weyl operator and its generalization,
Springer Science \& Business Media, 2012.

\bibitem {cogoni15}E. Cordero, M. de Gosson, F. Nicola, On the
invertibility of Born-Jordan quantization. J. Math. Pures Appl. 105 (2016) 537--557.

\bibitem {cogoni16a}E. Cordero, M. de Gosson, F. Nicola, On the reduction
of the interferences in the Born-Jordan distribution. Appl. Comput. Harmon.
Anal., in press. arXiv:1601.03719 [math.FA] (2016)

\bibitem {cogoni16b}E. Cordero, M. de Gosson, and F. Nicola, Time-frequency
Analysis of Born-Jordan pseudo-differential Operators, preprint
arXiv:1601.05303v1 [math.FA] (2016)

\bibitem {37}H. G. Feichtinger, Un espace de Banach de distributions
temp\'{e}r\'{e}es sur les groupes localement compact ab\'{e}liens, C. R. Acad.
Sci. Paris, S\'{e}rie A--B 290 (17) (1980) A791--A794.

\bibitem {38}H. G. Feichtinger, On a new Segal algebra, Monatsh. Math. 92(4) (1981) 269--289.

\bibitem {Birk}M. de Gosson, Symplectic Geometry and Quantum Mechanics.
Birkh\"{a}user, Basel, series \textquotedblleft Operator Theory: Advances and
Applications\textquotedblright\ (subseries: \textquotedblleft Advances in
Partial Differential Equations\textquotedblright), Vol. 166, 2006.

\bibitem {Birkbis}M. de Gosson, Symplectic Methods in Harmonic Analysis and in
Mathematical Physics, Birkh\"{a}user, 2011.

\bibitem {TRANSAM}M. de Gosson, Symplectic Covariance properties for Shubin
and Born-Jordan Pseudo-differential operators. Trans. Amer. Math. Soc.,
365(6) (2013) 3287--3307.

\bibitem {FP}M. de Gosson, Born-Jordan Quantization and the Equivalence of
the Schr\"{o}dinger and Heisenberg Pictures. Found. Phys. 44(10) (2014) 1096--1106.

\bibitem {Springer}M. de Gosson, Born-Jordan Quantization: Theory and
Applications, Springer 2016.

\bibitem {FPBJ}M. de Gosson, Born-Jordan quantization and the equivalence of
the Schr\"{o}dinger and Heisenberg pictures. Found. Phys. 44(10) (2014) 1096--1106.

\bibitem {PRBJ}M. de Gosson, From Weyl to Born-Jordan Quantization: The
Schr\"{o}dinger Representation Revisited, Phys. Reps. (2016), in press.

\bibitem {golu1}M. de Gosson, F. Luef, Preferred Quantization Rules:
Born-Jordan vs. Weyl; Applications to Phase Space Quantization. J.
Pseudo-Differ. Oper. Appl. 2(1) (2011) 115--139.

\bibitem {34}H. B. Domingo, E. A. Galapon, Generalized Weyl transform for
operator ordering:\ Polynomial functions in phase space. J.\ Math. Phys. 56 (2015)
022104.


\bibitem {65}K. Gr\"{o}chenig, Foundations of Time-Frequency Analysis,
Birkh\"{a}user, Boston, 2000.

\bibitem {69}L. H\"{o}rmander, The analysis of linear partial differential
operators I, Grundl. Math. Wissenschaft. 256, Springer, 1985.


\bibitem{lang} M. Langenbruch, Real roots of polynomials and right inverses for partial differential operators in the space of tempered distributions, Proc. Royal Soc. Edinburgh: Section A Mathematics, 114(3-4) (1990) 169--179.

\bibitem {75}R. G. Littlejohn, The semiclassical evolution of wave packets,
Phys. Rep. 138(4-5) (1986) 193--291.

\bibitem {mccoy}N. H. McCoy, On the function in quantum mechanics which
corresponds to a given function in classical mechanics, Proc. Natl. Acad. Sci.
U.S.A. 18(11) (1932) 674--676.

\bibitem {Shubin}M. A. Shubin, pseudo-differential Operators and Spectral
Theory, Springer-Verlag, (1987) [original Russian edition in Nauka, Moskva (1978)].

\bibitem {Treschev} D. V. Treschev, Quantum Observables: An Algebraic Aspect,
Proc. Steklov Inst. Math. 250 (2005) 211--244.

\bibitem {Treves}J.-F. Tr{e}ves, Introduction to Pseudodifferential and
Fourier Integral Operators, Springer Science \& Business Media, 1980.

\bibitem {Weyl}H. Weyl, Quantenmechanik und Gruppentheorie, Zeitschrift
f\"{u}r Physik 46 (1927).
\end{thebibliography}
\end{document}